\documentclass{elsart-preprint}
\makeatletter
\def\ps@pprintTitle{%
 \let\@oddhead\@empty
 \let\@evenhead\@empty
 \def\@oddfoot{\reset@font\hfil\thepage\hfil}
 \let\@evenfoot\@oddfoot
}
\makeatother

\usepackage{amsmath,amsfonts,amssymb,wasysym}
\usepackage[latin1]{inputenc}

\newcommand{\R}{{\mathbb R}}
\newcommand{\N}{{\mathbb N}}

\newcommand{\be}[1]{\begin{equation}\label{#1}}
\newcommand{\ee}{\end{equation}}
\renewcommand{\(}{\left(}
\renewcommand{\)}{\right)}
\newcommand{\ix}[2]{\int_{\R^{#1}}{#2}\;dx}
\newcommand{\idmu}[2]{\int_{\R^{#1}}{#2}\;d\mu}
\newcommand{\idmua}[2]{\int_{\R^d}{#2}\;d\mu_{#1}}
\newcommand{\iom}[1]{\int_{\Omega}{#1}\;dx}
\newcommand{\iomK}[1]{\int_{\Omega^{\mathrm K}}{#1}\;dx}
\newcommand{\nrmom}[2]{\|#1\|_{L^{#2}(\Omega)}}

\newcommand{\prf}{\par\smallskip\noindent{\sl Proof. \/}}
\newcommand{\finprf}{\unskip\null\hfill$\;\square$\vskip 0.3cm}
\newenvironment{proof}{\prf}{\finprf}
\renewcommand{\H}{\mathsf H}
\newcommand{\G}{\mathsf G}
\newcommand{\HH}{\gamma}
\newcommand{\gammaH}{H}

\newtheorem{theorem}{Theorem}
\newtheorem{corollary}[theorem]{Corollary}
\newtheorem{proposition}[theorem]{Proposition}
\newtheorem{remark}[theorem]{Remark}
\usepackage{color}
 
\begin{document}
\begin{frontmatter}
\title{Improved Poincar\'e inequalities}

\author[Dolbeault]{Jean Dolbeault}
\ead{dolbeaul@ceremade.dauphine.fr}
\ead[url]{www.ceremade.dauphine.fr/$\sim$dolbeaul}
\and
\author[Volzone]{Bruno Volzone}
\ead{bruno.volzone@uniparthenope.it}

\address[Dolbeault]{Ceremade (UMR CNRS no. 7534), Universit\'e Paris-Dauphine, Place de Lattre de Tassigny, 75775 Paris Cedex~16, France.}

\address [Volzone]{Universit\`a degli Studi di Napoli ``Parthenope'', Facolt\`a di Ingegneria, Dipartimento per le Tecnologie, Centro Direzionale Isola C/4 80143 Napoli, Italy.}

\begin{abstract}
\small Although the Hardy inequality corresponding to one quadratic singularity, with optimal constant, does not admit any extremal function, it is well known that such a potential can be improved, in the sense that a positive term can be added to the quadratic singularity without violating the inequality, and even a whole asymptotic expansion can be build, with optimal constants for each term. This phenomenon has not been much studied for other inequalities. Our purpose is to prove that it also holds for the gaussian Poincar\'e inequality. The method is based on a recursion formula, which allows to identify the optimal constants in the asymptotic expansion, order by order. We also apply the same strategy to a family of Hardy-Poincar\'e inequalities which interpolate between Hardy and gaussian Poincar\'e inequalities.
\end{abstract}

\end{frontmatter}

%%%%%%%%%%%%%%%%%%%%%%%%%%%%%%%%%%%%%%%%%%%%%%%%%%%%%%%%%%%%%%%%%%%%%%%%%%%%%%
%%%%%%%%%%%%%%%%%%%%%%%%%%%%%%%%%%%%%%%%%%%%%%%%%%%%%%%%%%%%%%%%%%%%%%%%%%%%%%
\section{Introduction}\label{Sec:Intro}

A considerable effort has been devoted to get improvements of Hardy inequalities. On $H^1_0(\Omega)$, we define the \emph{Hardy functional} by
\[
\H[u]:=\iom{|\nabla u|^2}-\frac 14\,(d-2)^2\iom{\frac{|u|^2}{|x|^2}}
\]
where $\Omega=\R^d$ or $\Omega$ is a bounded domain in $\R^d$ containing the origin, and $d\ge 3$. The standard \emph{Hardy inequality} asserts that
\be{Ineq:Hardy}
\H[u]\ge 0\quad\forall\;u\in H^1_0(\Omega)\;.
\ee
For an extension of \eqref{Ineq:Hardy} to a finite number of singularities, see \cite{MR2389919}. Inequality~\eqref{Ineq:Hardy} can be improved in various directions and we can list three lines of thought:
\begin{enumerate}
\item Prove that $\H[u]$ controls $\nrmom uq$ for some $q\in[2,2^*)$ with $2^*:=\!\!2\,d/(d-2)$, or $\nrmom{\nabla u}q$ for some $q\in[1,2)$. See \cite{alvino2009hardy} for a recent result in this direction, and \cite{MR1605678,MR1760280,MR817985,MR1938711,MR1742864} for earlier contributions.
\item Improve on the $\iom{\frac{|u|^2}{|x|^2}}$ term by showing that, with respect to the $1/|x|^2$ weight, not only $|u|^2$ is controlled, but also $|u|^2\log|u|^2$. See \cite{0906,DE2010,DETT2010} for recent papers in this direction.
\item Improve on the $1/|x|^2$ weight: see \cite{MR1918494,MR2462588,MR1970026,MR2048514,MR2091354,ghoussoub2007bessel,MR2443723,MR2509370}.
\end{enumerate}
A simple and well known method to establish \eqref{Ineq:Hardy} is based on an \emph{expansion of the square} which goes as follows. Let $u$ be a smooth function with compact support in $\Omega$ and observe that
\begin{align*}
0&\le\iom{|\nabla u+\tfrac{d-2}2\,\tfrac x{|x|^2}\,u|^2}\\
&=\iom{|\nabla u|^2}+\tfrac{(d-2)^2}4\iom{\frac{|u|^2}{|x|^2}}-\tfrac{d-2}2\iom{|u|^2\,\Big(\nabla\cdot\tfrac x{|x|^2}\Big)}=\H[u]
\end{align*}
where we have used an integration by parts and noted that $\nabla\cdot\frac x{|x|^2}=\tfrac{d-2}{|x|^2}$.

\medskip The Poincar\'e inequality with gaussian weight, or \emph{gaussian Poincar\'e inequality}, reads
\be{Ineq:Poincare}
\idmu d{|u-\bar u|^2}\le\idmu d{|\nabla u|^2}\quad\forall\;u\in H^1(\R^d,d\mu)
\ee
with $d\mu(x):=\mu(x)\,dx$, $\mu(x):=(2\pi)^{-d/2}\,e^{-|x|^2/2}$ and $\bar u:=\idmu du$. Our purpose is to study improvements of \eqref{Ineq:Poincare} in the spirit of what has been done for~\eqref{Ineq:Hardy}. Let us list some known results for \eqref{Ineq:Poincare}:
\begin{enumerate}
\item Spectral improvements are easily achieved under appropriate orthogonality conditions. See \cite{0904} for results and further references in this direction.
\item Replacing $|u|^2$ by $|u|^2\log|u|^2$ amounts to consider the logarithmic Sobolev inequality instead of the Poincar\'e inequality; see \cite{MR0420249} for an historical reference. There is a huge literature on this subject, which is out of the scope of the present paper.
\item A very standard argument based on the \emph{expansion of the square} has been repeatedly used in the literature. Let us give some details, in the gaussian case, as it is the starting point of our strategy.
\end{enumerate}

By expanding $\ix d{|\nabla(u\,e^{-|x|^2/4})|^2}$, we find that
\begin{equation}
\G[u]:=\idmu d{|\nabla u|^2}+\frac d2\idmu d{|u|^2}-\frac 14\idmu d{|x|^2\,|u|^2}\ge 0\;\label{Gaussian}.
\end{equation}
If $\bar u=0$, the middle term in $\G[u]$ can be estimated by \eqref{Ineq:Poincare}, thus showing that the following \emph{improved Poincar\'e inequality} holds:
\be{Ineq:Mouhot}
\idmu d{|x|^2\,|u|^2}\le 2\,(d+2)\idmu d{|\nabla u|^2}
\ee
(this inequality is an improvement in the sense that, as $|x|\to\infty$, the $|x|^2$ weight diverges). A slightly more general case has been considered for instance in \cite{Dolbeault2009511,DMS2010} (also see, \emph{e.g.}, \cite{mouhot2009fractional}). The expansion of the square method raises the following question. By \eqref{Gaussian} we know that \hbox{$\G[u]\ge 0$} for any $u\in H^1(\R^d,d\mu)$. With no additional assumption on $u$, \emph{is there a nonnegative function $W$ such that
\[
\G[u]\ge\idmu d{W\,|u|^2}
\]
for any $u\in H^1(\R^d,d\mu)$ and, if yes, can we give an asymptotic expansion as $|x|\to\infty$ of the best possible function $W\!$, order by order ?}

The purpose of this paper is to systematically investigate such improvements for gaussian Poincar\'e inequalities, following the same scheme as for the Hardy inequality. More precisely, using an elaborate \emph{expansion of the square} method, we derive an asymptotic expansion of the largest possible nonnegative function~$W$ and, order by order, find the best possible constants for any finite truncation of the asymptotic expansion.

To clarify our purpose, we will first recall in Section~\ref{Sec:Hardy} what can been done for the Hardy inequality and give a short proof of it based on the method used in \cite{MR2091354}. Then we shall adapt it to the gaussian Poincar\'e inequality, which provides us with our first main result: see Theorem~\ref{Thm:ImprovedPoincare} in Section~\ref{Sec:Poincaregaussian}. A striking parallel appears, which will be briefly investigated in Section~\ref{Sec:PoincareOther}, in the case of a family of inequalities interpolating between Hardy and Poincar\'e inequalities.

\medskip Before going further, let us quote a few additional references. Improvements of the Hardy inequality already have a quite long history. In \cite{MR1605678}, Brezis and V\'azquez have shown that in the case of a bounded domain $\Omega$, there exists a constant $\lambda_\Omega>0$ such that
\[
\lambda_\Omega\iom{|u|^2}\le\H[u]\quad\forall\;u\in H^1_0(\Omega)\;.
\]
The striking result of \cite{MR1862130}, by Adimurthi, Chaudhuri and Ramaswamy, is that a whole expansion in terms of iterated logarithms can be done close to the singularity (see also \cite{MR2184082} for a generalization in $W^{1,p}(\Omega)$). Filippas and Tertikas gave in \cite{MR1918494} the expression of the best constants for all terms of the expansion; also see \cite{MR2509370} for a more recent result, concerning $|u|^{p}$ in the remaining term, for the limit case $p=2^*$. In view of the generalization to relativistic models, Dolbeault, Esteban, Loss and Vega gave in \cite{MR2091354} an algebraic property which simplifies the computation of the expansion, while Ghoussoub and Moradifam in \cite{MR2443723} have established a rather simple characterization of the best constants. Some of the results of \cite{MR1918494} are summarized in Theorem~\ref{Thm:Filippas-Tertikas} below, with a simplified proof inspired by the combination of all above mentioned works. This proof will be a source of inspiration for the results on the gaussian Poincar\'e inequality, which are entirely new, and also for the Hardy-Poincar\'e inequality of Section~\ref{Sec:PoincareOther}.

%%%%%%%%%%%%%%%%%%%%%%%%%%%%%%%%%%%%%%%%%%%%%%%%%%%%%%%%%%%%%%%%%%%%%%%%%%%%%%
%%%%%%%%%%%%%%%%%%%%%%%%%%%%%%%%%%%%%%%%%%%%%%%%%%%%%%%%%%%%%%%%%%%%%%%%%%%%%%
\section{A key example: the improved Hardy inequality}\label{Sec:Hardy}

Let $r=|x|$ for any $x\in\Omega$, and set
\[
X_1(r) :=\(a-\log r\)^{-1}
\]
for some $a\ge1$ and
\[
X_k:=X_1\circ X_{k-1}
\]
for all $k\geq2$. We also define
\[
W_k:=\frac 14\,\prod_{j=1}^kX_j^2\quad\forall\;k\ge 1\;.
\]
We shall always assume that $0\in\Omega$. With $\delta_\Omega=\max_{\partial\Omega}|x|$, we choose $a=a_\Omega$ such that $a\ge1$ is the unique solution of $\delta_\Omega=1/(a-\log\delta_\Omega)$, \emph{i.e.} $a=\log\delta_\Omega+1/\delta_\Omega$, so that the interval $(0,\delta_\Omega]$ is stable under the action of $X_1$.

%------------------------------------------------------------------------------
\begin{theorem}\label{Thm:Filippas-Tertikas} Let $\Omega$ be a bounded domain containing the origin and assume that $a=a_\Omega$. With $W=\sum_{j=1}^\infty W_j$, we have
\be{Ineq:HardyImproved}
\iom{W\,\frac{|u|^2}{|x|^2}}\le\H[u]\quad\forall\;u\in H^1_0(\Omega)\;.
\ee
Moreover, such a function $W$ is optimal in the following sense. Assume that \eqref{Ineq:HardyImproved} holds for some nonnegative, bounded, radial function $W\!$. Then we have:
\begin{itemize}
\item[(i)] if $W$ converges as $r\to 0_+$ to some limit $\ell\in[0,+\infty]$ then $\ell=0$,
\item[(ii)] if $\lim_{r\to 0_+}W=0$ and if $\frac W{W_1}$ converges as $r\to 0_+$ to some limit $\ell_1\in[0,+\infty]$ then $\ell_1\le1$,
\item[(iii)] for any $N\ge2$ and with the convention $W_0:=0$, if
\[
\lim_{r\to 0_+}W=0\;,\quad\lim_{r\to 0_+}\frac{W-\sum_{j=0}^{k-1}W_j}{\sum_{j=1}^k W_j}=1\quad\forall\;k\in\{1,\,2\ldots N-1\}
\]
and if $\frac{W-\sum_{j=1}^{N-1}W_j}{\sum_{j=1}^{N}W_j}$ converges as $r\to 0_+$ to some limit $\ell_{N}\in[0,+\infty]$, then $\ell_{N}\le1$.

\end{itemize}
\end{theorem}
%------------------------------------------------------------------------------
The first part of Theorem~\ref{Thm:Filippas-Tertikas} has been obtained by Filippas and Tertikas in \cite[Theorem D, p. 190]{MR1918494} and the statement on the characterization of the best constants in the asymptotic expansion for all $W$ can be found in \cite[Theorem B', p. 192]{MR1918494}. Here we give a detailed proof based on the approach used in \cite{MR2091354}. Notice that if $W$ is not radially symmetric, some results can be recovered by applying Schwarz' symmetrization to $W(x)/|x|^2$.

Also notice that this expansion is independent of the value of $a$ used in the definition of $X_1$, as the behaviour of $W_j$ for $r$ close to $0_+$ does not depend on~$a$: what we have achieved is only an asymptotic expansion of the improvement~$W$ at the singularity.

\begin{proof} For simplicity, we split the proof in three steps.

\par\noindent\emph{Step 1.~Expansion of the square.\/} Suppose that $f=f(r) $ is a continuously differentiable function in an interval $[0,R]$ with $R>0$ such that $\Omega\subset B_R$, where $B_R$ denotes the ball of radius $R$ centered at the origin. Expanding the square $|\nabla u+f(r)\,\frac {x}{r^2}\,u|^2$ with $r=|x|$ and integrating by parts, we have
\[
0\leq\int_\Omega\left|\nabla u+f\,\frac{x}{r^2}\,u\right|^2\,dx=\int_\Omega|\nabla u|^2\,dx+\int _\Omega\(\frac{f^2}{r^2}-\frac{f'}r\,-\frac{d-2}{r^2}\,f\)|u|^2\,dx\;,
\]
that is
\[
\int_\Omega\(r\,f'+(d-2)\,f-f^2\)\frac{|u|^2}{|x|^2}\;dx\leq\int_\Omega|\nabla u|^2\,dx\;.
\]
Setting $g=f-(d-2)/2$, we get
\[
\frac 14\,(d-2)^2\int_\Omega\frac{|u|^2}{|x|^2}\;dx+\int_\Omega \frac{|u|^2}{|x|^2}\(r\,g'-g^2\)\,dx\leq\int_\Omega|\nabla u|^2\,dx\;.
\]
At this point, we observe that any bounded, positive solution on a neighborhood of $r=0_+$ of the equation
\[\label{Eqn:Positivity}
r\,g'-g^2=W\ge 0
\]
is such that $g(r)\le(a_0-\log r)^{-1}$ for some $a_0\in\R$ as $r\to0_+$ and, as a consequence, $\lim_{r\to 0_+}g(r)=0$. This proves that the constant $(d-2)^2/4$ in the expression of $\H[u]$ is optimal and proves Property (i).

\par\noindent\emph{Step 2.~Optimal behavior at first order in the asymptotic expansion.\/} It is worthwhile to notice that the function $r\mapsto X_1(r)=(a-\log r)^{-1}$ solves
\[
r\,g'-g^2=0\;.
\]
Also observe that $g(r)=\alpha\,X_1(r)$ solves
\[
r\,g'-g^2=(\alpha-\alpha^2)\,X_1^2\le\frac 14\,X_1^2
\]
with equality if and only if $\alpha=1/2$.

Let $h$ be such that $g(r)=X_1(r)\,h(s)$, with $s=-\log(X_1(r))$, so that $s\to+\infty$ as $r\to 0_+$. We claim that if
\[
\frac{W(r)}{W_1(r)}=4\,\frac{r\,g'(r)-g^2(r)}{X_1^2(r)}=4\(-h'(s)+h(s)-h^2(s)\)
\]
has a limit $\ell_1$ as $r\to0_+$, then $\ell_1\le1$. Let us prove it. If we have $\ell_1>1$, then
\[
-h'-\(h-\frac 12\)^2\sim\frac{\ell_1-1}4>0
\]
and then we know that $-\frac{h'}{(h-1/2)^2}\ge1$, so that for some constant $C$, we have
\[
\frac 1{h(s)-\frac12}>C+s
\]
for any $s$ large enough. This means that $\lim_{s\to\infty}h(s)=1/2$. Then we also know that
\[
h'\sim(1-\ell_1)/4{}<0\;,
\]
a contradiction. This proves (ii).

\par\noindent\emph{Step 3.~Induction.\/} Consider the sequence $(h_k)_{k\ge 1}$ of functions defined by
\[
h_1(r):=g(r)\;,\quad h_k(r)=t\,\big(h_{k+1}(t)+\tfrac 12\big)\;,\quad t=X_1(r)\in\(0,\delta_\Omega\)\;.
\]
An elementary computation shows that
\[
\frac{r\,h_k'(r)-h_k^2(r)}{t^2}-\frac 14=t\,h_{k+1}'(t)-h_{k+1}^2(t)\;.
\]
This implies that for all $k\geq1$ we find
\[\label{eq.2}
r\,g'(r)-g^2(r)=\sum _{j=1}^kW_j(r)+W_k(r)\big(z\,h_{k+1}'(z)-h_{k+1}^2(z)\big)
\]
with $z=X_k(r)$ and $r\in \(0,\delta_\Omega\)$. With this formula, it is clear that~\eqref{Ineq:HardyImproved} holds with $h_{k+1}=h_k$ for any $k\ge 1$, while proving the optimality of the constants in the asymptotic expansion goes at each iteration as in the computations of Step 2.\\ \end{proof}

As a consequence of Theorem~\ref{Thm:Filippas-Tertikas}, we also have an asymptotic expansion as $|x|\to\infty$ of an improved Hardy inequality. By the \emph{Kelvin transformation,} to any $u\in H^1_0(\Omega)$, we associate $v$ such that
\begin{equation}
v(x)=|x|^{2-d}\,u\big(|x|^{-2}\,x\big)\,,\label{Kelvintrans}
\end{equation}
where $v$ is defined on $\Omega^{\mathrm K}:=\left\{x\in\R^d\,:\, x/|x|^2\in\Omega\right\}$. By standard computations, we know that
\[
\iom{|\nabla u|^2}=\iomK{|\nabla v|^2}\;,\quad\iom{\frac{|u|^2}{|x|^2}}=\iomK{\frac{|v|^2}{|x|^2}}\;,
\]
and we can define $W_k^{\mathrm K}(r):=W_k(1/r)$ using the notations of Theorem~\ref{Thm:Filippas-Tertikas}, which can now be rewritten in the exterior domain $\Omega^{\mathrm K}$ as follows.%------------------------------------------------------------------------------
\begin {corollary}\label{Cor:Kelvin} Let $\Omega$ be a bounded domain containing the origin and assume that $a=a_\Omega$. With $W=\sum_{j=1}^\infty W_j^{\mathrm K}$, we have
\[\label{Ineq:HardyImprovedK}
\iomK{W\,\frac{|v|^2}{|x|^2}}\le\iomK{|\nabla v|^2}-\frac14\,(d-2)^2\iomK{\frac{|v|^2}{|x|^2}}\quad\forall\;v\in H^1_0(\Omega^{\mathrm K})\;.
\]
Moreover, such a function $W$ is optimal in the following sense. Assume that the above inequality holds for some nonnegative, radial function $W\!$. Then we have:
\begin{itemize}
\item[(i)] if $W$ converges as $r\to \infty$ to some limit $\ell\in[0,+\infty]$ then $\ell=0$,
\item[(ii)] if $\lim_{r\to \infty}W=0$ and if $\frac W{W_1^{\mathrm K}}$ converges as $r\to \infty$ to some limit $\ell_1\in[0,+\infty]$ then $\ell_1\le1$,
\item[(iii)] for any $N\ge2$ and with the convention $W_0^{\mathrm K}:=0$, if
\[
\lim_{r\to\infty}W=0,\quad\lim_{r\to\infty}\frac{W-\sum_{j=0}^{k-1}W_j^{{\mathrm K}}}{\sum_{j=1}^k W_j^{\mathrm K}}=1\quad\forall\;k\in\{1,\,2\ldots N-1\}
\]
and if $\frac{W-\sum_{j=1}^{N-1}W_j^{\mathrm K}}{\sum_{j=1}^{N}W_j^{\mathrm K}}$ converges as $r\to 0_+$ to some limit $\ell_{N}\in[0,+\infty]$, then $\ell_{N}\le1$.
\end{itemize}
\end{corollary}
%------------------------------------------------------------------------------

%%%%%%%%%%%%%%%%%%%%%%%%%%%%%%%%%%%%%%%%%%%%%%%%%%%%%%%%%%%%%%%%%%%%%%%%%%%%%%
%%%%%%%%%%%%%%%%%%%%%%%%%%%%%%%%%%%%%%%%%%%%%%%%%%%%%%%%%%%%%%%%%%%%%%%%%%%%%%
\section{Improved Poincar\'e inequality: the gaussian case}\label{Sec:Poincaregaussian}

Let us consider now the gaussian measure
\[
d\mu(x)=\mu(x)\,dx\;,\quad \mu(x)=\frac{e^{-|x|^2/2}}{(2\pi)^{d/2}}\;.
\]
Suppose that $d>2$ and define the functions
\[
t:=\frac 1{1+r^{d-2}}\;,\quad \delta:=-\frac t{\log(1-t)}\;,\quad X(t):=\frac{\log(1-t)}{\log(1-t)-1}\;.
\]

For any $r>0$ we have $t\in(0,1)$ and the functions $X$, $\delta$ are well defined. Besides, since $X(t)<1$ for any $t\in(0,1)$, we have that the interval $(0,1)$ is stable under the action of $X$. Moreover, for all $k\geq 0$, we set
\be{XYZW}\begin{array}{l}
X_0(t)=t\;,\quad X_{k+1}=X\circ X_k\;,\\
Y_0(t)=1\;,\quad Y_{k+1}=\(\delta\circ X_k\)^2=\frac {X^2_{k}(t)}{\log^2(1-X_{k}(t))}\,,\\
Z_0(t)=1\;,\quad Z_{k+1}(t)=X_k^2(t)\,\text{ for }k\geq0\;,\\
W_{k}(t)=Z_k(t)\mbox{ \thinspace if } k=0,1 \text{ and } W_k(t)=\Big(\;{\displaystyle\prod_{j=1}^{k-1}}Y_j(t)\Big)\,Z_k(t)\mbox{ if } k\ge 2\;.
\end{array}\ee

%------------------------------------------------------------------------------
\begin{theorem}\label{Thm:ImprovedPoincare} Suppose that $d\ge3$. With the above notations, \eqref{XYZW}, we have
\[
\G[u]\geq\frac14\,(d-2)^2\int_{\R^d}\frac{u^2}{|x|^2}\(\,\sum_{k=0}^\infty W_k(t)\)\,d\mu
\]
for any \hbox{$u\in H^1(d\mu)$}, where $t=1/(1+r^{d-2})$, $r=|x|$. Moreover, the expansion is asymptotically optimal, in the sense that at any order $N\geq 0$, if we consider an improved inequality of the form
\[
\G[u]\geq\frac14\,(d-2)^2\int_{\R^d}\frac{u^2}{|x|^2}\(\sum_{k=0}^NW_k(t)+\prod_{j=0}^NY_j(t)\,R_N(t)\)\,d\mu
\]
and if $R_N(t)$ converges as $t\to0$\,to some limit $\ell_N\in [0,\infty)$, then $\ell_N\leq 1$.\end{theorem}
%------------------------------------------------------------------------------
We may notice that no improvement can be achieved on the terms of order $|x|^2$ and $1$: if we had $\G[u]\ge\ell_{-2}\int_{\R^d}|u|^2\,|x|^2\,d\mu+\ell_{-1}\int_{\R^d}|u|^2\,d\mu$ for any $u\in H^1(d\mu)$, then testing the inequality with $u(x)=\exp(-(1-\varepsilon)\,|x|^2/4)$ shows that $\ell_{-2}\le0$ and $\ell_{-1}\le0$.

\begin{proof} As we did for Theorem~\ref{Thm:Filippas-Tertikas}, we split the proof in three steps.

\par\noindent\emph{Step 1.~Expansion of the square.\/} Let $g$ be any radial smooth function. Then, for any $u\in H^1\(d\mu\)$ by expanding the square $\int_{\R^d}|\nabla u+g(r)\,u\,x|^2\,d\mu$ and integrating by parts, we get
\[
\int_{\R^d}|u|^2\(r\,g'+d\,g-r^2\,(g^2+g)\)\,d\mu\leq\int_{\R^d}|\nabla u|^2\,d\mu\;.
\]
With the function $h$ defined by
\[
h(r):=r^2\(g(r)+\tfrac 12\)\,,
\]
we obtain a correction term to \eqref{Gaussian}, namely
\begin{multline*}
\int_{\R^d}\frac{|u|^2}{|x|^2}\(r\,h'+(d-2)\,h-h^2\)\,d\mu\\
\leq\int_{\R^d}|\nabla u|^2\,d\mu+\frac d2\int_{\R^d}|u|^2\,d\mu-\frac14\int_{\R^d}|u|^2\,|x|^2\,d\mu=\G[u]\;.
\end{multline*}
Consider the function $f$ such that
\[
(d-2)^2\,f(r):=r\,h'(r)+(d-2)\,h(r)-h^2(r)\;.
\]
The expansion of the square now amounts to
\[
(d-2)^2\int_{\R^d}\frac{|u|^2}{|x|^2}\,f(r)\,d\mu\le\G[u]\;.
\]
Our purpose is to identify the best possible function $f$.

\par\noindent\emph{Step 2.~Optimal behavior at zero order in the asymptotic expansion.\/}
Our goal is to maximize $f$ as $r\to+\infty$. Assume first that $4\,f(r)$ has a limit $\ell_0>1$ as $r\to\infty$. If we set $h(r)=(d-2)\,H(s)$ with $s=(d-2)\log r$, then $H$ solves
\[
H'+H-H^2\sim\frac{\ell_0}4\quad\mbox{as}\quad s\to\infty\;.
\]
For $s>0$, large enough, there exists $\varepsilon>0$ such that
\[
H'\ge\(H-\tfrac 12\)^2+ \varepsilon\ge \varepsilon\;,
\]
so that $\lim_{s\to\infty}H(s)=\infty$. But we can also write that $\frac{H'}{(H-1/2)^2}\ge 1$, so that, for some constant $C$,
\[
\frac1{H(s)-\frac12}<C-s\;,
\]
if $s$ is taken large enough. Thus we get $\lim_{s\to\infty}H(s)=\frac12$, a contradiction. On the other hand $H(s)=\frac12$ for any $s\in\R$ is admissible, thus proving that $\lim_{r\to\infty}f(r)=1/4$ can be obtained.

\par\noindent\emph{Step 3.~Induction.\/}
Observe that the nontrivial global solutions to the equation
\[
r\,h'+(d-2)\,h-h^2=0
\]
are given by $h(r)=\frac{d-2}{1+C\,(d-2)\,r^{d-2}}$ for an arbitrary constant $C$. This suggests to set
\[
t=t(r):=\frac 1{1+r^{d-2}}\;.
\]
If
\[
h(r)=(d-2)\,h_0(t)\;,
\]
then $f(r)$ can be rewritten in terms of $t$ as
\[
f(r)=-\,t\,(1-t)\,h_0'(t)+h_0(t)-h_0^2(t)\;.\label{eq.8}
\]
If $h_0(t)=\alpha\,t$ for some $\alpha\in\R$, then $f(r)=t^2\,\alpha\,(1-\alpha)$\,takes its largest possible value, namely $f(r)=t^2/4$, for $\alpha=1/2$. Now if
\[
h_0(t)=\frac{t}2+H_0(t)\;,
\]
we get
\[
f(r)-\frac{t^2}4=-\,t\,(1-t)\,H_0'(t)+(1-t)\,H_0(t)-H_0^2(t)\;.
\]

If we set $H_0(t):=\delta(t)\,h_1(s)$ where $s=X(t)$, then we have
\[
H_0'(t)=\delta'(t)\,h_1(s)+\delta(t)\,h_1'(s)\, X'(t)
\]
and by the definition of $X$ and $\delta$, it is not difficult to check that
\begin{equation}\label{firstit}
\frac{f(r)-\frac{t^2}4}{\delta^2(t)}=-\,s\,(1-s)\,h_1'(s)+h_1(s)-h_1^2(s)\;.
\end{equation}
Hence the r.h.s.~in \eqref{firstit} exactly takes the form of $f(r)$, with $t$ and $h_0$ replaced by $s$ and $h_1$ respectively. Since $\lim_{t\to 0}X(t)=0$, we can iterate this procedure.

Assume first that $W=\sum_{k=1}^\infty W_k$. By \eqref{XYZW} and \eqref{firstit}, we find that
\[
h_1=h_0\;.
\]
Hence, if we define $(R_k)_{k\ge0}$ by
\[
R_0(t):=4\, f(r)\quad\mbox{and}\quad R_{k+1}:=R_k\circ X_{k+1}\quad\mbox{for any}\quad k\geq0\;,
\]
then for any $N\geq 1$ we obtain
\begin{align*}
R_0(t) &=t^2+4\,\delta^2(t)\left[-\,s\,(1-s)\,h_1'(s)+h_1(s)-h_1^2(s)\right]\\
&=Z_1(t)+Y_1(t)\(R_0\circ X\)(t)\\
&=Z_1(t)+Y_1(t)\,R_1(t)\\
&=Z_1(t)+Y_1(t)\(Z_2(t)+Y_2(t)\,R_2(t)\)\\
&=Z_1(t)+Y_1(t)\,Z_2(t)+Y_1(t)\,Y_2(t)\,R_2(t)\\
&=\ldots=\sum_{k=1}^NW_k(t)+{\displaystyle\prod_{j=1}^N}Y_j(t)\,R_N(t)\;.
\end{align*}
Otherwise, we already know from Step 2 that
\[
R_0(t):=4\big[-t\,(1-t)\,h_0'(t)+h_0(t)-h_0^2(t)\big]
\]
is such that, if $\lim_{t\to0}R_0(t)=\ell_0$, then $\ell_0\le1$, and if $\ell_0=1$, then
\[
R_0(t)=Z_1(t)+Y_1(t)\,R_1(t)\quad\mbox{with}\quad R_1(t):=4\big[-\,t\,(1-t)\,h_1'(t)\,h_1(t)-h_1^2(t)\big]
\]
and the conclusion follows by a straightforward iteration.\end{proof}

Now we study the case $d=2$. First, set $t=1/\log r$. Then we let $a>1$ and define $R^\star=R^\star(a)=e^{1/t^\star}$ where $t^\star$ is given as the \textit{largest} positive solution of $t=X(t)$ with
\[
X(t):=\frac1{a-\log t}\;.
\]
Notice that $t=X(t)$ has a unique solution such that $t>1$. We also observe that $[0,t^\star]$ is stable under the action of $X$ (also see \cite{MR2091354} for further properties of $X$). Also notice that $t^\star>e^{a-1}$. With this new definition of $X$ and $\delta(t)=t$, $t=1/\log r$, we can now construct $X_k$, $Y_k=Z_k$ and $W_k$ as in~\eqref{XYZW}:
\be{XYZW2}\begin{array}{l}
X_0(t)=t\;,\quad X_{k+1}=X\circ X_k\;,\\
Y_{0}=1\;,\quad Y_{k+1}=X_k^2\,\text{ for }k\geq0\;,\\
W_0(t)=0\;,\quad \text{and } W_k(t)={\displaystyle\prod_{j=1}^k}Y_j(t)\mbox{ if } k\ge 1\;.
\end{array}\ee

%------------------------------------------------------------------------------
\begin{theorem}\label{Thm:GaussianDeux} Suppose that $d=2$. For any function $u\in H^1(d\mu)$ with support outside the ball of radius $R^\star$, we have
\[
\mathsf{G}[u]\geq \frac14\int_{\R^2}\frac{u^2}{|x|^2}\(\sum_{k=1}^\infty W_k(t)\)\,d\mu
\]
with $t=1/\log r$, $r=|x|$ and $W_k$ defined by \eqref{XYZW2}. Moreover, the expansion is asymptotically optimal, in the sense that at any order $N\geq 1$, if we consider an improved inequality of the form
\begin{equation*}
\mathsf{G}[u]\geq\frac14\int_{\R^2}\frac{u^2}{|x|^2}\(\sum_{k=0}^{N-1} W_k(t)+W_N(t)\, R_N(t)\)\,d\mu
\end{equation*}
and if $R_N(t)$ converges as $t\rightarrow 0$ to some limit $\ell_N\in [0,\infty)$, then $\ell_N\leq 1$.
\end{theorem}
%------------------------------------------------------------------------------
\begin{proof} The expansion of the square method reduces the problem to find the best possible function $f(r)=r\,h'(r)-h^2(r)$ such that
\[
\int_{\R^2}\frac{|u|^2}{|x|^2}\,f(r)\,d\mu
\leq\int_{\R^2}|\nabla u|^2\,d\mu-\frac14\int_{\R^2}|u|^2\,|x|^2\,d\mu+\int_{\R^2}|u|^2\,d\mu\;.
\]
If $f(r)\sim\ell\ge0$ as $r\to+\infty$, then it follows that $H(t)=h(r)$ with $t=\log r$ solves
\[
H'-H^2\sim\ell
\]
as $t\to\infty$. On the one hand, if $\ell>0$, then $H(t)\ge\frac\ell 2\,t$ as $t\to\infty$, \emph{i.e.}~$h(r)\ge\frac\ell 2\,\log r\to+\infty$, and on the other hand,
\[
\frac{h'}{h^2}\ge\frac 1r
\]
means that for some constant $C$ and for $r$ large enough, we have
\[
C-\frac 1{h(r)}\ge\log r
\]
which implies that $\lim_{r\to\infty}h(r)=0$, a contradiction. As a consequence, $\ell=0$. In other words, we have shown that the first term in the expansion (that is, the term of order $1/|x|^2$) is $W_0=0$.

The nontrivial solutions to the equation
\[
r\,h'-h^2=0
\]
are
\[
h=\frac1{C-\log r}\;,
\]
with $C$ being an arbitrary constant. If we set $t=1/\log r$ and consider $h_0$ such that $h(r)=h_0(t)$, we easily infer that
\[
f(r)=-\,t^2\,h_0'(t)-h_0^2(t)\;. \label{eq.3}
\]
If $h_0(t)=\alpha\,t$ for some $\alpha\in\R$, the largest possible value of $f(r)=-\,t^2\(\alpha^2+\alpha\)$ is $t^2/4$. It is achieved for $\alpha=-1/2$. Now if
\[
h_0(t)=-\frac t2+H_0(t)\;,
\]
we get
\[
f(r)-\frac{t^2}4=-\,t^2\,H_0'(t)+t\,H_0(t)-H_0^2(t)\;. \label{eq.4}
\]
As in the proof of Theorem~\ref{Thm:ImprovedPoincare}, if we set
\[
H_0(t)=t\,h_1(s)\quad\mbox{and}\quad s=X(t)\;,
\]
using the definition of $X$, it is easy to verify that
\[
\frac{f(r)-\,t^2/4}{t^2}=-s^2\,h_1'(s)-h_1^2(s)\;.
\]
Now it is enough to argue as in the proof of Theorem~\ref{Thm:ImprovedPoincare} to conclude.\end{proof}

As a concluding remark, we notice that we can combine the results of Theorems~\ref{Thm:ImprovedPoincare} and \ref{Thm:GaussianDeux} with the method used for proving~\eqref{Ineq:Mouhot} to get, for any $u\in H^1(\R^d,d\mu)$ such that $\bar u=\idmu du=0$,
\[
2\,(d+2)\int_{\R^d}|\nabla u|^2\,d\mu\ge\int_{\R^d}u^2\,|x|^2\,d\mu+(d-2)^2\sum_{k=0}^\infty\int_{\R^d}\frac{u^2}{|x|^2}\,W_k(t)\,d\mu
\]
if $d\ge3$ and $W_k$ is defined as in~\eqref{XYZW}, and, under the same assumptions as in Theorem~\ref {Thm:GaussianDeux},
\[
4\int_{\R^2}|\nabla u|^2\,d\mu\ge\int_{\R^2}u^2\,|x|^2\,d\mu+\sum_{k=1}^\infty\int_{\R^2}\frac{u^2}{|x|^2}\,W_k(t)\,d\mu
\]
if $d=2$ and $W_k$ is defined as in~\eqref{XYZW2}, thus improving also~\eqref{Ineq:Mouhot} for any $d\ge 2$.

%%%%%%%%%%%%%%%%%%%%%%%%%%%%%%%%%%%%%%%%%%%%%%%%%%%%%%%%%%%%%%%%%%%%%%%%%%%%%%
%%%%%%%%%%%%%%%%%%%%%%%%%%%%%%%%%%%%%%%%%%%%%%%%%%%%%%%%%%%%%%%%%%%%%%%%%%%%%%
\section{Improved Hardy-Poincar\'e inequalities}\label{Sec:PoincareOther}

%%%%%%%%%%%%%%%%%%%%%%%%%%%%%%%%%%%%%%%%%%%%%%%%%%%%%%%%%%%%%%%%%%%%%%%%%%%%%%
\subsection{Hardy-Poincar\'e inequalities}

In this section, we shall consider improvements of a family of \emph{Hardy-Poincar\'e inequalities} which has been investigated in \cite{BBDGV-CRAS,BBDGV,BDGV}. Let $h_\alpha(x):=(1+|x|^2)^\alpha$ and define $d\mu_\alpha(x)=h_\alpha(x)\,dx$, for any $\alpha\le0$. From \cite{BDGV}, we know that
\be{Eqn:Hardy-Poincare}
\Lambda_{\alpha,d}\idmua{\alpha-1}{|u-\mu_{\alpha-1}(u)|^2}\le\idmua{\alpha}{|\nabla u|^2}\quad\forall\;u\in H^1(\R^d,d\mu_\alpha)
\ee
with the convention
\[\begin{array}{ll}
\mu_{\alpha-1}(u):=\idmua{\alpha-1}u\quad&\mbox{if}\quad\alpha\in(-\infty,-(d-2)/2)\;,\\
\mu_{\alpha-1}(u):=0\quad&\mbox{if}\quad\alpha\in(-(d-2)/2,0)\;.
\end{array}\]
The inequality holds not only in $H^1(\R^d,d\mu_\alpha)$ but also in the larger space $\{u\in L^2(\R^d,d\mu_{\alpha-1})\,:\,\nabla u\in L^2(\R^d,d\mu_\alpha)\}$. This is easy to establish by density of smooth functions with support in $\R^d\setminus\{0\}$. The optimal value of $\Lambda_{\alpha,d}$ has been determined in \cite{BDGV}. If $d\ge 2$, we have
\begin{eqnarray*}
&&\Lambda_{\alpha,d}=-2\,\alpha\quad\mbox{if}\quad\alpha\in(-\infty,-d)\;,\\
&&\Lambda_{\alpha,d}=-2\,(d+2\,\alpha)\quad\mbox{if}\quad\alpha\in(-d,-(d+2)/2)\;,\\
&&\Lambda_{\alpha,d}=\tfrac 14\,(d-2+2\,\alpha)^2\quad\mbox{if}\quad\alpha\in(-(d+2)/2,0)\;.
\end{eqnarray*}
Notice that for $\alpha=-(d-2)/2$, we find $\Lambda_{\alpha,d}=0$ and the inequality fails. See~\cite{BGV} for more details in such a case. In the limit case $\alpha=0$, if we apply~\eqref{Eqn:Hardy-Poincare} to $u_\lambda(x)=\lambda^{d/2}\,u(\lambda\,x)$ and take the limit $\lambda\to 0_+$, we recover the Hardy inequality \eqref{Ineq:Hardy}. If we apply~\eqref{Eqn:Hardy-Poincare} to $u_\lambda(x)=u(\lambda\,x)$ with $\lambda=\sqrt{2\,|\alpha|}$ and take the limit $\alpha\to-\infty$, we recover the gaussian Poincar\'e inequality~\eqref{Ineq:Poincare}. Inequality~\eqref{Eqn:Hardy-Poincare} is therefore an interesting family of inequalities which interpolates between the Hardy inequality \eqref{Ineq:Hardy} and the gaussian Poincar\'e inequality~\eqref{Ineq:Poincare}. Our purpose is to show that the results of Sections~\ref{Sec:Hardy} and \ref{Sec:Poincaregaussian} can be adapted to this more general family of inequalities.

Let us take $\alpha<0.$ If we expand the square $\left|\nabla(u\,h_{\alpha/2})\right|^2$, an integration by parts gives
\begin{multline*}
0\leq\int_{\R^d}\left|\nabla(u\;h_{\alpha/2})\right|^2dx\\
=\int_{\R^d}|\nabla u|^2\,d\mu_\alpha+\alpha\,(2-\,\alpha)\int_{\R^d}\frac{|x|^2}{(1+|x|^2)^2}\,u^2\,d\mu_\alpha-\,\alpha\,d\int_{\R^d}u^2\,d\mu_{\alpha-1}\;,
\end{multline*}
that is
\[\label{HardPon1}
\int_{\R^d}|\nabla u|^2\,d\mu_\alpha-\,\alpha\,d\int_{\R^d}u^2\,d\mu_{\alpha-1}\geq\alpha\,(\alpha-2)\int_{\R^d}\frac{|x|^2}{(1+|x|^2)^2}\,u^2\,d\mu_\alpha\;.
\]
Exactly as in the Gaussian case, we can get the analogue of \eqref{Ineq:Mouhot}. By the Hardy-Poincar\'e inequality \eqref{Eqn:Hardy-Poincare}, if $\mu_{\alpha-1}(u)=0$, we can estimate the second term of the left-hand side by
\[
-\,\alpha\,d\int_{\R^d}u^2\,d\mu_{\alpha-1}\leq-\,\alpha\,d\,\Lambda_{\alpha,d}^{-1}\int_{\R^d}|\nabla u|^2\,d\mu_\alpha\;.
\]
Hence for all $u\in H^1(\R^d,d\mu_\alpha)$ such that $\;\mu_{\alpha-1}(u)=0$ we find
\begin{equation}
\int_{\R^d}\frac{|x|^2}{(1+|x|^2)^2}\,u^2\,d\mu_\alpha\leq\frac{1-\,\alpha\,d\,\Lambda_{\alpha,d}^{-1}}{\alpha\,(\alpha-2)}\int_{\R^d}|\nabla u|^2\,d\mu_\alpha\label{IMHP}
\end{equation}
which is an \emph{improved Hardy-Poincar\'e} inequality. Of course all these inequalities are valid if $\R^d$ is replaced by any open set $\Omega$ and the space $H^1(\R^d,d\mu_\alpha)$ by $H_0^1(\Omega,d\mu_\alpha)$.
Next, for any open set $\Omega$ and any $u\in H_0^1(\Omega,d\mu_\alpha)$, we define the functional
\[
\mathsf I_{\Omega} [u]:=\int_{\Omega}|\nabla u|^2\,d\mu_\alpha+\alpha\,(2-\,\alpha)\int_{\Omega}\frac{|x|^2}{(1+|x|^2)^2}\,u^2\,d\mu_\alpha-\,\alpha\,d\int_{\Omega}u^2\,d\mu_{\alpha-1}
\]
and we know that $\mathsf I_{\Omega}[u]\geq0$ for any $u\in H_0^1(\Omega,d\mu_\alpha)$.

%%%%%%%%%%%%%%%%%%%%%%%%%%%%%%%%%%%%%%%%%%%%%%%%%%%%%%%%%%%%%%%%%%%%%%%%%%%%%%
\subsection{A scheme for improving Hardy-Poincar\'e inequalities}\label{Sec:scheme}

To get a full asymptotic expansion, the strategy is similar to the one used in Theorems~\ref{Thm:Filippas-Tertikas},~\ref{Thm:ImprovedPoincare} and~\ref{Thm:GaussianDeux}, but various cases have to be distinguished depending on the dimension.

\par\noindent\emph{Step 1.~Expansion of the square.\/} Let $g$ be any smooth radial function on $\R^d$. For any $u\in H^1( \R^d,d\mu_\alpha) $, if we expand the square $\left|\nabla u+g(r) u\;x\right|^2$ and integrate by parts with respect to the measure $d\mu_\alpha$, we find
\begin{align*}
0 & \leq\int_{\R^d}\left|\nabla u+g(r)\,u\,x\right|^2\,d\mu_\alpha\\
&=\int_{\R^d}|\nabla u|^2\,d\mu_\alpha+\int_{\R^d} \Big[-r\,g'-\frac{(2\alpha+d)\,r^2+d}{1+r^2}\,g+r^2\,g^2\Big]\,u^2\,d\mu_\alpha\;.\nonumber
\end{align*}
Define now a function $h(r)$ by
\[
h(r)=(1+r^2)\,g(r)-\,\alpha\;.
\]
We find that
\begin{equation}
\mathsf I_{\R^d}[u]\geq\int_{\R^d}f(r)\,u^2\,\d\mu_{\alpha-2}
\label{HardPon4}
\end{equation}
where
\begin{equation}
f(r):=(1+r^2)\,r\,h'+\left[(d-2)\,r^2+d\right]h-r^2\,h^2\,.\label{f}
\end{equation}

The nontrivial positive global solutions to the equation
\begin{equation}
(1+r^2)\,r\,h'+\left[(d-2)\,r^2+d\right]h-r^2\,h^2=0\label{eqdiff}
\end{equation}
are given for $d\ge3$ by
\[
h(r)=(d-2)\,\frac{1+r^2}{r^2+C(d-2)r^d}\sim\frac{1}{C\,r^{d-2}}\quad\mbox{as}\quad r\to+\infty
\]
where $C$ is an arbitrary positive constant, while the positive solutions when $d=2$ are given in a neighborhood of $r=0_+$ by
\[
h(r)=\frac{1+r^2}{r^2\(C-\log r\)}
\]
for some $C\in\R$.

\par\noindent\emph{Step 2.~Optimal behavior at zeroth order in the asymptotic expansion.\/} Our aim is now to maximize $f(r)/r^2$ as $r\to+\infty$. To do that, assume that
\[
\frac{f(r)}{r^2}-d\,\frac h{r^2}=\frac{1+r^2}r\,h'+(d-2)\,h-h^2
\]
has a limit $\frac\ell4\,(d-2)^2$ if $d\ge3$ and $\ell$ if $d=2$.

Assume first that $d\ge 3$. With $h(r)=(d-2)\,\gammaH(s)$ and $s=\frac{d-2}2\log(1+r^2)\to+\infty$, we find that
\[
\gammaH'(s)+\gammaH(s)-\gammaH^2(s)\sim\frac\ell4
\]
and get a contradiction if $\ell>1$, by the same arguments as in Theorems~\ref{Thm:ImprovedPoincare}. As a consequence, $\limsup_{r\to+\infty}h(r)\le(d-2)/2$ and $f(r)\sim\frac 14\,(d-2)^2\,\ell\,r^2$ with $\ell\le1$ as $r\to+\infty$. Finally, it  is straightforward to check that if $f(r)/r^2$ has a limit larger than $(d-2)^2/4$ as $r\to+\infty$, then $h(r)\sim C\,r^2$ up to a positive constant $C$ and we also get a contradiction.

If $d=2$, we can work as in Theorem~\ref{Thm:GaussianDeux}. If $\ell>0$ and also get a contradiction using $h(r)=\gammaH(s)$ and $s=\frac12\log(1+r^2)\to+\infty$. After some elementary considerations, this shows that as $r\to+\infty$, the limit of $f(r)/r^2$ is non-positive if it exists.

\par\noindent\emph{Step 3.~Induction.\/} Assume temporarily that for some functions $\HH_d$, $\delta_d$ and $X$ to be determined, we have
\begin{align}
&X_0(t)=t\;,\quad X_{k+1}=X\circ X_k\;\nonumber,\\
&Y_0(t)=1\;,\quad Y_{k+1}=\(\delta_d\circ X_k\)^2,\label{XYZWIHP}\\
&Z_{0}(t)=1\text{ for }d\geq3\text{ and }Z_{0}(t)=0\text{ for }d=2\;,\quad Z_{k+1}=\HH_d\circ X_{k}\text{ for }k\geq0,\nonumber\\
&W_{k}(t)=Z_k(t)\mbox{ \thinspace if } k=0,1 \text{ and } W_k(t)={\displaystyle\prod_{j=1}^{k-1}}Y_j(t)\,Z_k(t)\quad\mbox{if}\quad k\ge 2\;.\nonumber
\end{align}
The functions $\gamma_d$ and $\delta_d$ are determined as follows. With $t=1/\log r$ if $d=2$, $t=r^{2-d}$ if $d\ge3$ and $h(r)=h_0(t)$, with $c_2=1$ and $c_d=(d-2)^2$ if $d\ge3$, we may write
\begin{equation}
\frac{f(r)}{c_d\,r^2}=\mathcal F(t,h_0(t),h_0'(t))\label{0it}
\end{equation}
for some function $\mathcal F$. The above choice of $t$ is justified by the behavior for $r\rightarrow\infty$ of the solutions $h$ to equation \eqref{eqdiff}. Then we identify a function $\gamma_d$ such that
\begin{enumerate}
\item[(i)] for some constant $\beta\in\R$, $\mathcal F(t,h_0(t),h_0'(t))=\beta\,\HH_d(t)+o(\HH_d(t))$ as $t\to0$,
\item[(ii)] if $\beta$ takes its largest possible value, then we look for some function $h_1$ and $s=X(t)$ such that
\begin{equation}
\mathcal F(t,h_0(t),h_0'(t))-\beta\,\HH_d(t)=\delta_d(t)^2\,\mathcal F(s,h_1(s),h_1'(s))\label{1it}\;.
\end{equation}
\end{enumerate}
The functions $X\,,\delta_{d}$ will be chosen in order to satisfy a relation of the type
\begin{equation}
\mathcal{A}(t,X(t),X^{\prime}(t),\delta_{d}(t))=\mathcal{B}(t,X(t),\delta_{d}(t),\delta^{\prime}_{d}(t))=\delta_{d}^{2}(t)\label{abdelt}
\end{equation}
where $\mathcal{A}\,,\mathcal{B}$ are suitable functions. Then the analogue of Theorems~\ref{Thm:Filippas-Tertikas},~\ref{Thm:ImprovedPoincare} and~\ref{Thm:GaussianDeux} holds. The main difficulty is to build the functions $\HH_d$ and $X$. A restriction comes from the requirement that some interval is stable under the action of $X$. This program can be completed in dimension $d=2$, $3$ and $4$, and also in dimension higher than $4$, in exterior domains.

%%%%%%%%%%%%%%%%%%%%%%%%%%%%%%%%%%%%%%%%%%%%%%%%%%%%%%%%%%%%%%%%%%%%%%%%%%%%%%
\subsection{The case $d=2$}

If $h(r)=h_0(t)$ with $t=1/\log r$, we get that
\begin{equation}
\frac{f(r)}{r^2}=-\,(1+e^{-\frac 2t})\,t^2\,h_0'(t)+2\,e^{-\frac 2t}\,h_0(t)-h_0^2(t)\;\label{cas2}.
\end{equation}
Implicitly define the function $h_1$ such that
\[
h_0(t):=-\beta\,t+\delta_2(t)\,h_1(s)
\]
where $\beta\in (0,+\infty)$ and $\delta_2(t) $, $s=X(t)$ are two functions to be determined, and
\begin{equation}
\HH_2(t):=(1+e^{-\frac 2t})\,t^2-2\,e^{-\frac 2t}\,t-\beta\,t^2\;.\label{gamma2}
\end{equation}
Replacing $h_{0}(t)$ in \eqref{cas2} we find
\begin{multline*}
\frac{f(r)}{r^2}-\beta\,\HH_2(t)\\
=-\,(1+e^{-\frac 2t})\,t^2\,\delta_2(t)\,X'(t)\,h_1'(s)\hspace*{5cm}\\
+\left[2\(e^{-\frac 2t}-\beta\,t\)\delta_2(t)-\,(1+e^{-\frac 2t})\,t^2\,\delta_2'(t)\right] h_1(s)-\delta_2^2(t)\,h_1^2(s)\;.\nonumber
\end{multline*}
We can then write
\[
\frac{f(r)}{r^2}-\beta\,\HH_2(t)=-\,A(t)\,(1+e^{-\frac 2s})\,s^2\,h_1'(s)+2\,e^{-\frac 2{s}}\,B(t)\,h_1(s)-C(t)\,h_1^2(s)
\]
where we have set
\[
\mathcal{A}:=\dfrac{(1+e^{-\frac2{t}})}{(1+e^{-\frac 2X})}\,\dfrac{t^2\,\delta_2\,X'}{X^2}\;,\quad \mathcal{B}:=\dfrac{2\,(e^{-\frac 2t}-\beta\,t)\,\delta_2-(1+e^{-\frac 2t})\,t^2\,\delta_2'}{2\,e^{-\frac 2{X}}}
\]
and $\mathcal{C}:=\delta_2^2$. We look for functions $X$ and $\delta_2$ such that $\mathcal{A}=\mathcal{B}={C}$, \emph{i.e.} satisfying equations \eqref{abdelt}. This amounts to
\be{deltaduedim}
\delta_2=\dfrac{(1+e^-{\frac 2t})}{(1+e^{-\frac 2X})}\,\dfrac{t^2\,X'}{X^2}
\ee
and
\[
\frac{X'}{(1+e^\frac 2X)\,X^2}=-\frac{\delta_2'}{2\,\delta_2}+\frac{1-\beta\,t\,e^\frac 2t}{t^2\,(1+e^\frac 2t)}\;.\label{casoduedim2}
\]
By taking the logarithmic derivative of equation \eqref{deltaduedim}, we obtain
\begin{equation*}
\frac{\delta_2'}{\delta_2}=\frac 2{t^2\,(1+e^\frac 2t)}-\frac{2\,X'}{(1+e^\frac 2X)\,X^2}+\frac 2t-2\,\frac{X'}X+\frac{X''}{X'}
\end{equation*}
so that $X$ solves the ordinary differential equation
\[
\frac{X''}{X'}-2\,\frac{X'}X=-\frac{2\,(\beta+1)}t+\frac{2\,\beta}{t\,(1+e^\frac 2t)}
\]
which leads to
\[
\log \frac{X'}{X^2}=-2\,(\beta+1)\,\log t+2\,\beta\int_1^t\frac{ds}{s\,(1+e^\frac 2s)}+C_1
\]
for some constant $C_1\in\R$. The solution with initial condition $X(0)=0$ can be written as
\[
X(t)=\left[C_0-e^{C_1}\int_1^ts^{-2\,(\beta+1)}\,\exp\(2\,\beta\int_1^s\frac{d\sigma}{\sigma\,(1+e^\frac 2 \sigma)}\)ds\right]^{-1}\,
\]
for some positive constant $C_0$. We also notice that for $t>0$ small enough and $\beta\in(0,1-1/e^2]$, the function $\gamma_2$ defined by \eqref{gamma2} is positive.  Notice that $t=1/\log r$ ranges in $(0,+\infty)$ as $r$ ranges in $(1,+\infty)$, and so we can take any $u$ supported outside the unit ball without further precautions. We remark that \emph{if} $\beta>1/2$, then $X(t)\sim t^{2\,\beta-1}$ as $t\to 0_+$, so that $X$ satisfies the initial condition $X(0)=0$. Besides, for a fixed $t^{\star}>1$, the constants $C_0$ and $C_1$ are chosen such that $X(t^{\star})=t^{\star}$, in order that the interval $[0,t^{\star}]$ is stable under the action of $X$.
%------------------------------------------------------------------------------
\begin{proposition}\label{Prop:Dim2HP} Assume that $d=2$ and $\beta\in(1/2,1-1/e^2]$. With the above notations and $\{W_k\}_k$ defined by \eqref{XYZWIHP}, for any $u\in H^1(\R^2,d\mu_\alpha) $, compactly supported outside the ball of radius $R=e^{1/t^{\star}}$ with $t^{\star}>1$, if $t=1/\log|x|$, then we have
\[
\mathsf I_{\R^2}[u]\geq\beta\int_{\R^2}\(\;\sum_{k=0}^{\infty}W_{k}(t)\)\,|x|^2\,u^2\,d\mu_{\alpha-2}\;.
\]
\end{proposition}
%------------------------------------------------------------------------------
At this point, optimality is clearly an open question because of the $\beta$ factor.

%%%%%%%%%%%%%%%%%%%%%%%%%%%%%%%%%%%%%%%%%%%%%%%%%%%%%%%%%%%%%%%%%%%%%%%%%%%%%%
\subsection{The case $d\ge 3$}
Assume that
\begin{equation*}%\HH_3(t):=\frac14\,\sqrt{t}\,\(10\,t^2-\sqrt{t}+2\)\quad\mbox{and}%
\HH_d(t):=t^\frac d{d-2}\,\(\tfrac2{d-2}-\tfrac14\,t^{\frac{d-4}{d-2}}\)\mbox{ if }d\ge4\;.\label{gamma4}
\end{equation*}
Since the function $\gamma_{d}$ for $d=3$ is positive only for $r\leq8$ and we want an asymptotic expansion for $r\rightarrow\infty$, we need a different choice of $\gamma_{3}$. More precisely, we set
\begin{equation*}
\HH_3(t):=\frac14\,\sqrt{t}\,\(10\,t^2-\sqrt{t}+2\).\label{gamma3}
\end{equation*}
The function $\gamma_{d}$ for $d\geq3$ will be the ``remaining term'' in \eqref{1it} coming out by plugging the two different expressions of $h_{0}(t)$ for $d=3$ and $d\geq4$ in \eqref{0it}, as we will see later. Moreover, we
let for all $d\geq3$
\be{functions}
t=\frac1{r^{d-2}}\;,\quad X(t)=\int_0^te^{-\frac12\nu_d(s)}\,ds\quad\mbox{and}\quad\delta_d(t)=\frac tX\,\frac{1+t^\frac 2{d-2}}{1+X^\frac 2{d-2}}\,X'
\ee
with
\[
\nu_3(s):=\int_0^s\frac{d\sigma}{\sqrt\sigma\,(1+\sigma^2)}\quad\mbox{and}\quad\nu_d(s):=\frac12\int_0^s\frac{d\sigma}{1+\sigma^\frac 2{d-2}}\mbox{ if }d\ge4\;.
\]
Notice that
\[
\nu_3(s)=\frac1{\sqrt2}\left[ \arctan\(1\!+\!\sqrt{2s}\)-\arctan\(1\!-\!\sqrt{2s}\) +\frac12\log\(\frac{1+\sqrt{2s}+s}{1-\sqrt{2s}+s}\) \right]\,.
\]
By definition of $X$, $X(t)\leq t$ for any $t\geq0$. The sequence $\left\{X_{k}\right\}_{k}$ is therefore decreasing, and in particular $X_{k}(t)\leq t$ for all $k$. Now we look for the sets where the functions $W_{k}$ are nonnegative. First, we notice that $\HH_d(t)$ is always positive if $d=3$, $4$.
%------------------------------------------------------------------------------
\begin{theorem}\label{Theorem5} Let $d=3$ or $4$, and assume that $\{W_k\}_k$ is defined by \eqref{XYZWIHP}. For any function $u\in H^1(\R^{d},d\mu_\alpha) $ we have
\[
\mathsf I_{\R^{d}}[u]\geq\frac14\,(d-2)^2\int_{\R^{d}}\(\;\sum_{k=0}^{\infty}W_{k}(t)\)|x|^2\,u^2\,d\mu_{\alpha-2}\;.
\]
Moreover, the expansion is asymptotically optimal, in the sense that at any order $N\geq 0$, if we consider an improved inequality of the form
\[
\mathsf I_{\R^{d}}[u]\geq\frac14\,(d-2)^2\int_{\R^{d}}\(\;\sum_{k=0}^{N}W_{k}(t)+\prod_{j=0}^{N}Y_{j}(t)\,R_N(t)\)\,|x|^2u^2\,d\mu_{\alpha-2}
\]
and if $R_N(t)$ converges as $t\to0$\,to some limit $\ell_N\in [0,\infty)$, then $\ell_N\leq 1$.
\end{theorem}
%------------------------------------------------------------------------------
\begin{proof}
Let us set $h(r)=(d-2)\,h_0(t)$ with $t=r^{2-d}$. Then
\begin{equation}
\frac{f(r)}{(d-2)^2\,r^2}:=-\(1+t^\frac 2{d-2}\)t\,h_0'(t)+\(1+\tfrac{d}{d-2}\,t^\frac 2{d-2}\) h_0(t)-h_0^2(t) \label{HardPon5}\;.
\end{equation}
Next consider the function $h_1$ implicitly defined by
\[
h_0(t)=\left\{\begin{array}[c]{ll}
\frac{t}4+\delta_d(t)\,h_1(s)\quad&\text{if }d=4\\[4pt]
\frac14\,\sqrt{t}+\delta_d(t)\,h_1(s)\quad&\text{if }d=3
\end{array}
\right.
\]
where $\delta_d=\delta_d(t) $ and $s=X(t)$ are two positive functions to be determined. Replacing in \eqref{HardPon5}, we find that
\begin{multline*}
\frac{f(r)}{(d-2)^2\,r^2}-\frac14\,\HH_d(t)\\
\begin{array}{ll}=&\displaystyle-\,\frac tX\,\frac{1+t^\frac 2{d-2}}{1+X^\frac 2{d-2}}\,\delta_d\,X'\,(1+s^\frac 2{d-2})\,s\,h_1'(s)\smallskip\\
&\displaystyle+\,\tfrac{(2d-4+2d\,t^\frac 2{d-2}-(d-2)\,t)\,\delta_d-2(d-2)\,(1+t^\frac 2{d-2})\,t\,\delta_d'}{2(d-2+d\;X^\frac 2{d-2})}\(1+\tfrac{d}{d-2}\,s^\frac 2{d-2}\)h_1(s)\\
&\displaystyle-\,\delta_d^2\,h_1^2(s)\;.\end{array}
\end{multline*}
As a consequence, we get an expression that is similar to \eqref{HardPon5}, namely
\begin{align*}
& \frac1{\delta_d^2(t) }\left[\frac{f(r)}{(d-2)^2\,r^2}-\frac14\,\HH_d(t)\right]\label{prova}\\
&=-\(1+s^{\frac2{d-2}}\) s\,h_1^{\prime}(s) +\(1+\tfrac{d}{d-2}\,s^\frac 2{d-2}\)\,h_1(s) -h_1^2(s) \nonumber
\end{align*}
by imposing that $X$ and $\delta_d$ satisfy equations \eqref{abdelt}, where
\begin{align*}
&\mathcal{A}:=\frac tX\,\frac{1+t^\frac 2{d-2}}{1+X^\frac 2{d-2}}\,\delta_d\,X'\\&\mathcal{B}:=\displaystyle\,\frac{(2d-4+2d\,t^\frac 2{d-2}-(d-2)\,t)\,\delta_d-2(d-2)\,(1+t^\frac 2{d-2})\,t\,\delta_d'}{2(d-2+d\;X^\frac 2{d-2})}.
\end{align*}
It is not difficult to show that the functions $X$, $\delta_d$ in \eqref{functions} are solutions to \eqref{abdelt}, satisfying the conditions $X(0)=0$ and $\lim_{t\to 0_+}\delta_d(t)=1$.

At this point, we can iterate all the arguments. Indeed, as in the proof of Theorem~\ref{Thm:ImprovedPoincare}, if we set
\[
R_0(t):=\frac{f(r)}{(d-2)^2\,r^2}
\]
and define
\[
R_{k+1}:=R_k\circ X_{k+1}\quad\mbox{for}\quad k\geq0
\]
by \eqref{XYZWIHP}, for any $N\geq 1$ we obtain
\begin{align*}
R_0(t) &=\frac14Z_1(t)+Y_1(t)\,R_1(t)\\
&=\ldots=\frac14\sum_{k=1}^NW_k(t)+{\displaystyle\prod_{j=1}^N}Y_j(t)\,R_N(t)\;.
\end{align*}
\end{proof}

%------------------------------------------------------------------------------
\begin{remark} If $d=3$, $4$, we can combine the results of Theorem~\ref{Theorem5} with inequality \eqref{Eqn:Hardy-Poincare} to get a further improvement of \eqref{IMHP}, that reads as
\begin{equation*}
\int_{\R^d}|\nabla u|^2\,d\mu_\alpha\geq\frac1{1-\,\alpha\,d\,\Lambda_{\alpha,d}^{-1}}\int_{\R^d}\left[\alpha\,(\alpha-2)+\tfrac{(d-2)^2}4\sum_{k=0}^{\infty}W_{k}(t)\right]|x|^2\,u^2\,d\mu_{\alpha-2}
\end{equation*}
for any $u\in H^1(\R^{d},d\mu_\alpha)$ such that $\mu_{\alpha-1}(u)=0$.\end{remark}
%------------------------------------------------------------------------------

When $d\ge5$, if \hbox{$\zeta:=(8/\!(d-2))^\frac{d-2}{d-4}$}, we remark that for $t\leq\zeta$ we get $Z_1(t)=\HH_d(t)\geq0$ and for all $k\geq1$ it follows that $X_{k}(t)\leq X_{k}(\zeta)\leq\zeta$ thus $Z_{k+1}(t)=\HH_d(X_{k}(t))\geq0$. Therefore we define $R:=(1/\zeta)^{1/d-2}$, so that for any $r\geq R$ we know that $Z_{k}(t)\geq0$ for all $k\in\N$. Hence, for $d\ge 5$, only external domains should be considered.

%%%%%%%%%%%%%%%%%%%%%%%%%%%%%%%%%%%%%%%%%%%%%%%%%%%%%%%%%%%%%%%%%%%%%%%%%%%%%%
\subsection{External domains with $d\ge5$}

As in Section~\ref{Sec:Hardy}, consider the Kelvin transformation given by \eqref{Kelvintrans}. Let $\Omega$ be an open set of $\mathbb{R}^{d}$ containing the origin and for any function $u\in C_0^1\(\Omega\)$ we consider $v(y)=|y|^{2-d}\,u(x)$ with $y=x/|x|^2$. It follows that $v\in C_0^1\(\Omega^{K}\) $. By standard calculations, we find
\[
|\nabla v|^2=|y|^{-2d}\,|\nabla u(x)|^2+(d-2)^2\,|y|^{2-2d}\,u^2(x)+(d-2)\,|y|^{-2d}\(y\cdot\nabla u^2(x)\)
\]
with $y=x/|x|^2$. Integrating and performing an integration by parts to the last term in the expression of $|\nabla v|^2$, we get that
\[
\int_{\Omega^{K}}|\nabla v|^2\,d\mu_\alpha=\int_{\Omega}\frac{|\nabla u|^2}{|x|^{2\alpha}}\;d\mu_\alpha+ 2\,\alpha\,(d-2)\int_{\Omega}\frac{u^2}{|x|^{2\(\alpha +1\)}}\;d\mu_{\alpha-1}\label{functI1}
\]
so that, if $\int_{\Omega}|x|^{-\(d+2\alpha\)}u\,d\mu_{\alpha-1}=0$, by applying \eqref{Eqn:Hardy-Poincare} to $v$ we get
\[
\int_{\Omega}\frac{|\nabla u|^2}{|x|^{2\alpha}}\;d\mu_\alpha \geq \left[\Lambda_{\alpha,d}-2\,\alpha\,(d-2)\right]\int_{\Omega}\frac{u^2}{|x|^{2(\alpha+1)}}\;d\mu_{\alpha-1}\;,\label{functI3}
\]
which is a new weighted Hardy-Poincar\'e inequality. The constant is positive because $\alpha$ is negative and $d\ge 2$.

With the change of variables $y=x/|x|^2$ and $u$ given in terms of $v$ by \eqref{Kelvintrans}, we find that $\mathsf I_{\Omega^K}[v]=\mathsf J_{\Omega}[u]$ where
\begin{multline*}
\mathsf J_{\Omega}[u]:=\int_{\Omega}|\nabla u|^2\,|x|^{-2\alpha}\,d\mu_\alpha+\alpha\,(d-4)\int_{\Omega}\frac{u^2}{|x|^{2(\alpha+1)}}\,d\mu_{\alpha-1}\\+\alpha\,(2-\alpha) \int_{\Omega}\frac{u^2}{|x|^{2(\alpha+1)}}\,d\mu_{\alpha-2}
\end{multline*}
is nonnegative. This leads to the inequality
\begin{equation}
\int_{\Omega}|\nabla u|^2\,|x|^{-2\alpha}\,d\mu_\alpha\geq\alpha(\alpha-2)\int_{\Omega}\frac{u^2}{|x|^{2(\alpha+1)}}\,d\mu_{\alpha-2}\;.\label{HarPonsin}
\end{equation}
For dimension $d\geq5$ we have seen that the terms $W_{k}(t)$ in the asymptotic expansion are nonnegative out of the ball $B_{R}$, with $R=((d-2)/8)^{(d-2)/(d-4)}$. Since $B_{1/R}^{K}=\R^d\setminus B_{R}$, the Kelvin transformation allows then to add a whole asymptotic expansion as $|x|\rightarrow0$ at the right hand side of \eqref{HarPonsin}. Indeed, we have the following
%------------------------------------------------------------------------------
\begin{corollary}\label{Corollary8}
If $d\ge5$ and $B_{1/R}$ denotes the ball of radius $1/R$, then we have
\[
\mathsf J_{B_{1/R}}[u]\geq\frac14\,(d-2)^2\int_{B_{1/R}}\(\;\sum_{k=0}^{\infty}W_{k}(|x|^{d-2})\)\frac{u^2}{|x|^{2(\alpha+1)}}\;d\mu_{\alpha-2}\label{d>4}
\]
for any \hbox{$u\in H_0^1(B_{1/R},d\mu_\alpha)$}, and the asymptotic expansion as $|x|\rightarrow0$ at the right hand side is optimal again, in the sense specified in Theorem 6.
\end{corollary}
%------------------------------------------------------------------------------
\begin{proof}
If $u\in H_0^1(B_{1/R},d\mu_\alpha)$, let $v\in H_0^1(B_{1/R}^{K},d\mu_\alpha)$ given by \eqref{Kelvintrans}. Therefore we can use the same method in the proof of Theorem 6, up to replace $\R^{d}$ by $B_{1/R}^{K}$. Then we find
\[
\mathsf I_{B_{1/R}^{K}}[v]\geq\int_{B_{1/R}^{K}}f(r)\,v^{2}\,d\mu_{\alpha-2}
\]
with the same $f$ defined in \eqref{f}. Clearly the optimality arguments shown in the step 2 of section 4.2 still holds. At the end, we obtain
\[
\mathsf I_{B_{1/R}^{K}}[v]\geq\frac14\,(d-2)^2\int_{B_{1/R}^{K}}\(\;\sum_{k=0}^{\infty}W_{k}(t)\)|x|^2\,v^2\,d\mu_{\alpha-2}\;.
\]
and the Kelvin transformation \eqref{Kelvintrans} implies the desired result.
\end{proof}

%%%%%%%%%%%%%%%%%%%%%%%%%%%%%%%%%%%%%%%%%%%%%%%%%%%%%%%%%%%%%%%%%%%%%%%%%%%%%%
%%%%%%%%%%%%%%%%%%%%%%%%%%%%%%%%%%%%%%%%%%%%%%%%%%%%%%%%%%%%%%%%%%%%%%%%%%%%%%
\section{Concluding remarks and open questions}\label{Sec:Conclusion}

The Poincar\'e inequality (with gaussian weight) is a \emph{spectral gap} inequality and it is easy to obtain improved constants by imposing constraints on the set of functions. The orthogonality with respect to all Hermite polynomials of order less than $k$ will automatically increase the value of the corresponding Rayleigh quotient. This has been investigated for instance in \cite{0904} (also see references therein) in connection with other interpolation inequalities. A similar approach has also been developed for Hardy-Poincar\'e inequalities: see \cite{DT} in case of the measure $d\mu_\alpha$ introduced in Section~\ref{Sec:PoincareOther}. The links between Hardy-Poincar\'e inequalities and the gaussian Poincar\'e ineuquality do not stop here: indeed, if we use the scaling $u_{\lambda}(x)=\lambda^{\frac{d-2}2}\,u(\lambda\,x)$, where $\lambda=\sqrt{2\,|\alpha|}$, we get
\[
\mathsf I_{\R^{d}}[u_{\lambda}]\rightarrow\mathsf G[u]
\]
as $\alpha\rightarrow -\infty$ and the equation we have seen in the proof of Theorem~\ref{Theorem5}
\begin{equation}
(1+r^2)\,r\,h'+\left[(d-2)\,r^2+d\right]h-r^2\,h^2=0\label{eqth5}
\end{equation}
degenerates into the equation used in the proof of Theorem~\ref{Thm:ImprovedPoincare},
\begin{equation}
r\,h'+(d-2)\,h-h^2=0\;.\label{eqth3}
\end{equation}
To see this last property, it is enough to notice that if $h$ is a solution to \eqref{eqth5}, then for any $\lambda>0$ we have
\[
(1+\lambda^4\,r^2)\,\lambda^2\,r\,h'(\lambda^2\,r)+\left[(d-2)\,\lambda^4\,r^2+d\right]h(\lambda^2\,r)-\lambda^4\,r^2\,h^2(\lambda^2\,r)=0
\]
and by making the change of variable $s=\lambda\,r$, it is straightforward to get
\[
\(\frac1{\lambda^2}+s^2\)s\,\frac{d}{ds}[h(\lambda\,s)]+\left[(d-2)\,s^2+\frac{d}{\lambda^2}\right]\,h(\lambda\,s)-s^2\,h^2(\lambda\,s)=0
\]
hence, if $\lambda=\sqrt{2\,|\alpha|}$, assuming that $h(\lambda\,s)\rightarrow \widetilde{h}(s)$ as $\alpha\rightarrow-\infty$, \textit{at least formally} we obtain that $\widetilde{h}$ solves \eqref{eqth3}.

In this paper, we have given improvements on the potential (characterized by its asymptotic expansion as either $|x|\to 0$ or $|x|\to\infty$) without imposing additional conditions on the set of functions. However, as noticed in the introduction, by requiring $\bar u=0$, we get the improved inequality \eqref{Ineq:Mouhot}. In that case the measure is not the same for the $L^2$ term and for the Dirichlet energy. This raises the interesting question of combining both approaches which, as far as we know, is a completely open issue.

Improvements have been achieved as a series of positive terms that can be added to the weight in the $L^2$ norm controlled by the Dirichlet form, thus leaving the inequality as a comparison between two quadratic functionals. The optimal additional terms are obtained by iterating a map involving some logarithmic terms. As mentioned in the introduction, there are other improvements which amount to control $u^2\,\log u^2$ terms by the Dirichlet form: the logarithmic Sobolev inequality and the logarithmic Hardy inequality, for instance. Is it possible to relate these two approaches ?

The basic tool of our approach is a simple expansion of a square. However, by leaving the weight undetermined, we obtain a non-local integro-differential equation which allows to identify the best possible growth order by order, and build an induction scheme. Some care is however required when defining the class of potentials under consideration. \emph{Optimality} of improvements of inequalities is a delicate matter which deserves further studies, if one wishes to relax some of our assumptions. However, let us mention as a final comment that one of the advantages of the expansion of a square is that, in our framework, optimality cases are easy to identify and it is not the less remarkable aspect of our results that the computation of optimal constants is then straightforward in most of the cases.

%%%%%%%%%%%%%%%%%%%%%%%%%%%%%%%%%%%%%%%%%%%%%%%%%%%%%%%%%%%%%%%%%%%%%%%%%%%%%%
%%%%%%%%%%%%%%%%%%%%%%%%%%%%%%%%%%%%%%%%%%%%%%%%%%%%%%%%%%%%%%%%%%%%%%%%%%%%%%
\bigskip\noindent{\it Acknowledgments.\/}{ \small The authors acknowledge support by the ANR-08-BLAN-0333-01 project CBDif-Fr.}

\par\noindent{\scriptsize\copyright\,2011 by the authors. This paper may be reproduced, in its entirety, for non-commercial purposes.}

%%%%%%%%%%%%%%%%%%%%%%%%%%%%%%%%%%%%%%%%%%%%%%%%%%%%%%%%%%%%%%%%%%%%%%%%%%%%%%
%%%%%%%%%%%%%%%%%%%%%%%%%%%%%%%%%%%%%%%%%%%%%%%%%%%%%%%%%%%%%%%%%%%%%%%%%%%%%%
%\bibliographystyle{siam}
%\small
%\bibliography{Poincare}

\end{document}